\documentclass[11pt]{amsart}
\usepackage{palatino}
\usepackage{amsfonts}
\usepackage{amssymb}
\usepackage{amscd}
\usepackage{xypic}

\usepackage{amsmath}
\usepackage{mathrsfs}
\usepackage{graphicx}
\usepackage{epstopdf}
\newtheorem{theorem}{Theorem}[section]
\newtheorem{proposition}[theorem]{Proposition}

\newtheorem{lemma}[theorem]{Lemma}
\theoremstyle{definition}
\newtheorem{definition}[theorem]{Definition}
\newtheorem{remark}[theorem]{Remark}

\newtheorem{conjecture/question}[theorem]{Conjecture/Question}

\newtheorem{remark/definition}[theorem]{Remark/Definition}
\newtheorem{terminology/notation}[theorem]{Terminology/Notation}

\def\PP{{\textbf P}}
\def\OO{\mathcal{O}}

\def\P{\mathcal{P}}

\def\cM{\mathcal{M}}

\def\cC{\mathcal{C}}
\def\H{\mathcal{H}}

\def\mm{\overline{\mathcal{M}}}
\def\hh{\overline{\mathcal{H}}}

\newcommand{\Mbar}[2]{\overline{\mathcal{M}}_{{#1}, {#2}}}

\begin{document}
\title{On the Kodaira dimension of Hurwitz spaces}
\author[G. Farkas]{Gavril Farkas}

\address{Humboldt-Universit\"at zu Berlin, Institut f\"ur Mathematik,  Unter den Linden 6
\hfill \newline\texttt{}
 \indent 10099 Berlin, Germany} \email{{\tt farkas@math.hu-berlin.de}}

\author[S. Mullane]{Scott Mullane}

\address{Goethe-Universit\"{a}t Frankfurt, Institut f\"ur Mathematik, Robert-Mayer-Str. 6-8, \hfill
 \newline \indent 60325 Frankfurt am Main, Germany}
 \email{{\tt
mullane@math.uni-frankfurt.de}}

\maketitle

\begin{center}
\emph{To Olivíer Debarre, with friendship.}
\end{center}

\vskip 8pt

The Hurwitz space $\H_g^k$ is the parameter space of covers $[f\colon C\rightarrow \PP^1, p_1, \ldots, p_b]$, where $C$ is a smooth algebraic curve of genus $g$ and $f$ is a degree $k$ map simply branched over $b=2g+2k-2$ distinct points $p_1, \ldots, p_b\in \PP^1$. Note that we  choose an \emph{ordering} of the branch points of $f$. The origins of the interest in Hurwitz spaces go back to Riemann's Existence Theorem and they have been used by Clebsch \cite{Cl} and Hurwitz \cite{Hu}, as well as much later in \cite{HM} to derive important information on the moduli space $\cM_g$ of curves of genus $g$.   We denote by $\hh_{g}^k$ the moduli space of admissible covers constructed by Harris and Mumford \cite{HM}, whose study has been further refined  in \cite{ACV} via twisted stable maps. It comes equipped with two  maps
$$\xymatrix{
  & \hh_{g}^k \ar[dl]_{\mathfrak{b}} \ar[dr]^{\sigma} & \\
   \mm_{0,b} & & \mm_{g}       \\
                 }$$
where $\mathfrak{b}$ associates to an admissible cover its (ordered) set of branch points,  whereas
$\sigma$ assigns to an admissible cover the stable model of its source curve. The symmetric group $\mathfrak{S}_{b}$ operates on $\hh_g^k$ by permuting the branch points of each admissible cover and we set $\hh_{g,k}:=\hh_g^k/\mathfrak{S}_b$. Recall that the Kodaira-Iitaka dimension of a normal $\mathbb Q$-factorial projective variety $X$ is defined as the Iitaka dimension of its canonical bundle. We say that the Kodaira-Iitaka dimension of $X$ is maximal if it equals $\mbox{dim}(X)$.

Our first result concerns the Kodaira-Iitaka dimension of the \emph{stack} $\overline{H}_g^k$ of degree $k$ admissible covers for which we have optimal results:

\begin{theorem}\label{kodhurw}
The Kodaira-Iitaka dimension of $\overline{H}_g^k$ is maximal for every  $g\geq 2$ and $k\geq 3$.
\end{theorem}

Our result, which is uniform in $g$ and $k$, is sharp. When $k=2$ the map $\mathfrak{b}$, while being ramified along the boundary at the level of stacks, induces an isomorphism between the coarse moduli spaces $\hh_g^2$ and $\mm_{0,2g+2}$. In particular, $\hh_g^2$ is a rational variety for every $g$ and the canonical class of both the stack $\overline{H}_g^2$, as well as that of the coarse moduli space $\hh_g^2$ is not effective.

\vskip 3pt

A crucial aspect in the proof of Theorem \ref{kodhurw} is played by the map
$$\theta \colon \hh_g^k\rightarrow \mm_{g,b+b[k-2]},$$
which associates to an admissible cover $\bigl[f\colon C\rightarrow R, p_1, \ldots, p_b\bigr]$ the pointed curve $$[C, x_1, \ldots, x_b, A_1, \ldots, A_b],$$ where $x_i\in f^{-1}(p_i)$ is the unique ramification point of $f$ lying over the branch point $p_i$ and $A_i:=f^{-1}(p_i)-\{x_i\}$ is the $i$-th set of \emph{antiramification points} of $f$, that is, the set of residual points in the fibre over the $i$-th branch point of $f$. The moduli space $\mm_{g,b+b[k-2]}$ is defined as a suitable quotient of $\mm_{g,b(k-1)}$ by the finite group $\mathfrak{S}_{k-2}^b$, the action being given by permuting $b$ subsets of $k-2$ marked points, we refer to  Section 2 for details. On $\mm_{g,b+b[k-2]}$ we consider the effective divisor $\widetilde{\mathfrak{D}}$ as being the closure  of the locus of those pointed curves $[C, x_1, \ldots, x_b, A_1, \ldots, A_b]$ for which there exists a subset $S$ consisting of $g$ ramification or antiramification points of $f$ such that
$$h^0\Bigl(C,\OO_C\bigl(\sum_{x\in S} x\bigl)\Bigr)\geq 2.$$ The divisor $\widetilde{\mathfrak{D}}$ has two desired features. On the one hand its class has a negative coefficient of its Hodge class, on the other hand, the number of marked points being so large (and this is the point in involving the antiramification points as well) the (positive) coefficient of the cotangent classes corresponding to the marked points in the class $[\widetilde{\mathfrak{D}}]$ is relatively small. Taking advantage of these features, in Section 3 we prove Theorem~\ref{kodhurw} by finding a positive constant $B>0$ such that the class $K_{\overline{H}_g^k}-B\cdot \theta^*(\widetilde{\mathfrak{D}})$ can be expressed as a boundary divisor on $\hh_g^k$, in which the coefficient of each irreducible component of $\partial\hh_g^k$ is \emph{positive}. As we then point out in  Remark \ref{bigclasses} this implies the bigness of the canonical class $K_{\overline{H}_g^k}$ of the stack of admissible covers.

\vskip 3pt

Next we move to the coarse moduli space $\hh_g^k$ and in this paper we restrict ourselves to the case of trigonal curves, for which we prove the following result:

\begin{theorem}\label{kodhurwcoarse3}
The moduli space $\hh_g^3$ has maximal Kodaira-Iitaka dimension for all $g\geq 2$.
\end{theorem}

Theorem \ref{kodhurwcoarse3} follows the argument used in proving Theorem \ref{kodhurw}, once we observe that the big boundary representative of the canonical class $K_{\overline{H}_g^3}$ of the stack of trigonal curves is sufficiently positive to offset the negative coefficient of the ramification divisor of the map $\overline{H}_g^3\rightarrow \hh_g^3$, therefore it produces a big boundary representative of the canonical class of $\hh_g^3$ as well.

We stress that in Theorems \ref{kodhurw}, \ref{kodhurwcoarse3}  we have results on the Kodaira-Iitaka dimension of the stack, and respectively, the coarse moduli space of the space of admissible covers. In the case of $\mm_g$ where the boundary has an extremely simple structure, the Kodaira dimension of the stack and that of the coarse moduli space trivially coincide, but this is no longer necessarily the case for the Hurwitz space which has a complicated boundary structure. We explain in Proposition \ref{cancoarse} the relation between the canonical class of $\overline{H}_g^k$ and that of $\hh_g^k$.

\vskip 4pt

Moving to the case of covers of high degree, when $k\geq \frac{g+2}{2}$ one has a generically finite map $\chi \colon \hh_g^k\rightarrow \mm_{g,2k-g-2}$ obtained by attaching to an admissible cover $$\bigl[f\colon C\rightarrow R, \ p_1, \ldots, p_{b}\bigr]$$ the stabilization of the nodal $(2k-g-2)$-pointed curve $\bigl[C, q_1, \ldots, q_{2k-g-2}]$, where $q_i\in C$ is the unique ramification point of $f$ over the branch point $p_i$. It follows that $\hh_{g}^k$ is of general type whenever the Kodaira dimension of $\mm_{g, 2k-g-2}$ is maximal. This is the case for all $g\geq 22$ and we refer to Proposition \ref{kodram} for a precise statement.

\vskip 3pt

For the Hurwitz space $\hh_{g,k}$ where the branch points are unordered one cannot expect a uniform result in the style of Theorem \ref{kodhurw}. Indeed, it has been classically known that the (unordered) Hurwitz spaces $\hh_{g,k}$ are unirational for all $g$ as long as $k\leq 5$. These results have been extended to the case of $6$-gonal covers for finitely many cases by Geiss \cite{G}. Further unirationality results have been obtained in \cite{ST}, whereas some isolated examples of Hurwitz spaces $\hh_{g,k}$ with effective canonical class in the range when the Kodaira dimension of $\mm_g$ is unknown have been produced in \cite{F} and \cite{FR}. Using that $\mm_g$ is of general type for $g\geq 22$ (see \cite{HM}, \cite{EH} and \cite{FJP}), it immediately follows that $\hh_{g,k}$ (and therefore $\hh_g^k$ as well) is of general type when $\frac{g+2}{2}\leq k\leq g+1$. On the other hand, when $k\geq g+2$, then $\hh_{g,k}$ is birational to a  projective bundle over a universal Picard variety, therefore it is uniruled.

\section{Divisors on Hurwitz spaces}

The main actor of this paper is the stack $\overline{H}_g^k$ of \emph{twisted stable maps} into the classifying stack $\mathcal{B} \mathfrak{S}_{k}$ of the symmetric group $\mathfrak S_{k}$. Precisely, we set
$$\overline{H}_g^k:=\overline{M}_{0,b}\Bigl(\mathcal{B} \mathfrak S_{k}\Bigr),$$
where $b:=2g+2k-2$. We denote by $\hh_g^k$ the associated coarse moduli space. The stack $\overline{H}_g^k$ is the \emph{normalization} of the stack of admissible covers introduced by Harris and Mumford in \cite{HM} and which, for lack of better notation, we denote by $\overline{HM}_g^k$. A point in $\hh_g^k$ corresponds to a twisted stable map
$[f\colon C\rightarrow R, p_1, \ldots, p_{b}]$, where  $C$ is a nodal curve of arithmetic genus $g$, the target curve $R$ is a tree of smooth rational curves, $f$ is a finite map of degree $k$ satisfying $f^{-1}(R_{\mathrm{sing}})=C_{\mathrm{sing}}$, and $p_1, \ldots, p_{b}\in R_{\mathrm{reg}}$ denote the branch points of $f$. Note that the branch points $p_1, \ldots, p_b$ are ordered. Moreover, the two ramification indices of $f$ on the two branches of $C$ over each singularity of $C$ coincide. The extra information distinguishing $[f \colon C\rightarrow R, p_1, \ldots, p_b]$ from its underlying admissible cover is the stacky data at each of the points in $C_{\mathrm{sing}}$. The  \emph{branch} morphism
$$\mathfrak{b}\colon \hh_g^k\rightarrow \mm_{0, b},$$ assigns to $[f\colon C\rightarrow R, p_1, \ldots, p_{b}]$ the stable $b$-pointed curve $[R, p_1, \ldots, p_b]$ of genus $0$. Clearly, $\mathfrak{b}$ is a finite map. Its degree, which has been computed classically by Hurwitz \cite{Hu} for $k\leq 6$,  has been recently the object of much attention in Gromov-Witten theory.  We also have a regular map
$$\sigma\colon \hh_g^k\rightarrow \mm_{g}$$
which assigns to $[f\colon C\rightarrow R, p_1, \ldots, p_{b}]$ the stable model of the nodal curve $C$.

\vskip 3pt

In what follows, we discuss the geometry of the boundary divisors of $\hh_g^k$. For $i=0, \ldots, \frac{b}{2}$, let $B_i$ be the boundary divisor of $\mm_{0, b}$ defined as the closure of the locus of unions of two smooth rational curves meeting at one point, such that precisely $i$ of the marked points lie on one component. A boundary divisor of $\hh_{g}^k$ is determined by the following data:
\begin{enumerate}
\item A partition $I\sqcup J=\{1,\dotsc, b\}$, with $|I|\geq 2$ and
$|J|\geq 2$.
\item Transpositions $\{w_i\}_{i\in I}$ and $\{w_j\}_{j\in J}$ in $\mathfrak S_k$, satisfying $$\prod_{i\in I} w_i = u, \ \ \prod_{j\in J}w_j=u^{-1}.$$
\end{enumerate}
We denote by $\mu:=(m_1, \ldots, m_{\ell})\vdash k$ be the partition corresponding to the cycle type of the element $u\in \mathfrak S_k$ appearing above.  Furthermore, we set
\begin{equation}\label{mu}
m(\mu):=\mathrm{lcm}\bigl(m_1, \ldots, m_{\ell}\bigr) \ \ \mbox{ and }  \ \ \frac{1}{\mu}:=\frac{1}{m_1}+\cdots+\frac{1}{m_{\ell}}.
\end{equation}

\begin{definition}
For $i=2, \ldots, \frac{b}{2}$ and a partition $\mu$ of $k$, let $E_{i:\mu}$ be the boundary divisor on $\hh_g^k $ given as the closure of the locus of covers $\bigl[f\colon C\rightarrow R, \ p_1, \ldots, p_{b}\bigr]\in \hh_g^k$,
where $[R=R_1\cup_p R_2,  p_1, \ldots, p_b]\in B_{|I|}\subseteq \mm_{0, b}$, with $f^{-1}(p)$ having partition type $\mu$, and exactly $i$ of the branch points $p_1, \ldots, p_{b}$ lying on the component $R_1$.
\end{definition}

The linear independence of the classes $[E_{i:\mu}]\in CH^1(\hh_g^k)$ has been established in \cite{P}. Note that it is often the case that $E_{i:\mu}$ splits into several irreducible components. All the Chow groups we consider are with rational coefficients. In particular, we identify
$CH^1(\hh_g^k)$ and $CH^1(\overline{H}_g^k)$ and the class $[E_{i:\mu}]\in CH^1(\hh_g^k)$ refers to the stacky $\mathbb Q$-class of the corresponding boundary divisor.

\subsection{The local structure of $\overline{H}_g^k$.} Over the stack $\overline{H}_g^k$ of twisted stable maps we consider the universal
degree $k$ admissible cover $f\colon \cC \rightarrow \P$, where
\begin{equation}\label{eq:univcurve}
\P:=\overline{H}_g^k\times_{\overline{M}_{0, b}} \overline{M}_{0, b+1}
\end{equation} is the universal
degree $k$ \emph{orbicurve} of genus zero over $\overline{H}_g^k$.
We fix a general point
$$t=[f \colon C \rightarrow R,  p_1, \ldots, p_{b}]$$ of a boundary divisor
$E_{i:\mu}$, where
$\mu=(m_1, \ldots, m_{\ell})$ is a partition of $k$.  In particular, $R$ is the union of two smooth rational curves $R_1$ and $R_2$ meeting at a point $p$. The local ring at $t$ of the stack $\overline{HM}_g^k$ of Harris-Mumford admissible covers has
the following local description, see \cite[p. 62]{HM}:
\begin{equation}\label{localring2}
\widehat{\mathcal{O}}_{t, \overline{HM}_g^k}\cong    \mathbb C
    [[t_1, \ldots, t_{b-3}, s_1, \ldots,
    s_{\ell}]]/s_1^{m_1}=\cdots=s_{\ell}^{m_{\ell}}=t_1,
  \end{equation}
where $t_1$ is the local parameter on $\overline{\mathcal{M}}_{0, b}$ corresponding to
smoothing the node $p\in~R$.  We set $f^{-1}(p)=\{q_1, \ldots, q_{\ell}\}$, with $f$ being ramified with order $m_j$ at $q_j$, for $j=1, \ldots, \ell$. The local ring of $\cC$ at the point $[t,q_j]$ is $\widehat{\mathcal{O}}_{t, \overline{HM}_g^k}[[x_j, y_j]]/x_j y_j=s_j$, while the local ring of $\P$ at the point $[t,p]$ is  $\widehat{\mathcal{O}}_{t, \overline{HM}_g^k}[[ u_j, v_j]]/u_j v_j=t_1$. The map $\cC\rightarrow \P$ is given in local coordinates by
$$u_j=x_j^{m_j}, v_j=y_j^{m_j}, \mbox{ for } j=1, \ldots, \ell,$$
in particular $s_1^{m_1}=\cdots=s_{\ell}^{m_{\ell}}=t_1$.  In order to determine the local ring of $\overline{H}_g^k$ at the point $t$ one normalizes
the ring (\ref{localring2}). We introduce a further parameter $\tau$ and choose primitive $m_j$-th roots of unity $\zeta_j$ for $j=1, \ldots, \ell$. These choices correspond to specifying the stack structure of the cover $f\colon C\rightarrow R$ at the points of $C$ lying over $p\in R_{\mathrm{sing}}$. Thus
\begin{equation}\label{localringtwist}
\widehat{\mathcal{O}}_{[t, \zeta_1, \ldots, \zeta_{\ell}], \ \overline{H}_g^k}=\mathbb C [[t_1, \ldots, t_{b-3}, \tau]]
\end{equation}
and  $s_j=\zeta_j  \tau^{\frac{m(\mu)}{m_j}}$, for $j=1, \ldots, \ell$. Accordingly,  the map $\mathfrak{b}\colon \overline{H}_g^k\rightarrow \overline{M}_{0, b}$ (at the level of stacks!), being given locally by $t_1=\tau^{m(\mu)}$, it is branched with order $m(\mu)$ at each point $[t, \zeta_1, \ldots, \zeta_{\ell}]\in E_{i:\mu}$.

\vskip 3pt

This discussion summarizes how the boundary divisors on $\mm_{0,b}$ pull-back under the finite map $\mathfrak{b}\colon \overline{H}_g^k\rightarrow \overline{M}_{0,b}$, see also \cite[p. 62]{HM}, or \cite[Lemma 3.1]{GK}:
\begin{equation}\label{pb}
\mathfrak{b}^*(B_i)=\sum_{\mu\vdash k} m(\mu)E_{i:\mu}.
\end{equation}

\subsection{The Hodge class on the compactified Hurwitz space.} By definition, the Hodge class on $\hh_g^k$ is pulled back from $\mm_g$ via the map $\sigma$. Its class $\lambda:=\sigma^*(\lambda)$ on $\hh_g^k$ has been determined first in \cite{KKZ} using Bergman kernel methods. An algebro-geometric proof, using Grothedieck-Riemann-Roch, appeared in \cite[Theorem 1.1]{GK}. The Hodge class on $\overline{H}_g^k$ has the following expression in terms of boundary classes:
\begin{equation}\label{hodgecl}
\lambda=\sum_{i=2}^{g+k-1} \sum_{\mu\vdash k} m(\mu)\left(\frac{i(2g+2k-2-i)}{8(2g+2k-3)}-\frac{1}{12}\Bigl(k-\frac{1}{\mu}\Bigr)\right)[E_{i:\mu}]\in CH^1(\hh_g^k).
\end{equation}

For a given $i$, the sum (\ref{hodgecl}) is taken over those partitions $\mu$ of $k$ corresponding to conjugacy classes of  permutations that can be written as products of $i$ transpositions.  We pick an admissible cover $$[f\colon C=C_1\cup C_2\rightarrow R=R_1\cup_p R_2,\  p_1, \ldots, p_{b}]\in \mathfrak{b}^*(B_2),$$ and set $C_1:=f^{-1}(R_1)$ and $C_2:=f^{-1}(R_2)$ respectively. Note that the curves $C_1$ and $C_2$ may well be disconnected.

\vskip 4pt

We record the following well-known facts on $\mm_{0,b}$, see for instance \cite{AC2}:
\begin{proposition}\label{mob}
(i) One has the following formulas in $CH^1(\mm_{0,b})$:
$$K_{\mm_{0, b}}= \sum_{i=2}^{\lfloor \frac{b}{2}\rfloor } \left(\frac{i(b-i)}{b-1}-2\right)[B_i] \ \ \mbox{  and } \ \ \kappa_1=\sum_{i=2}^{\lfloor \frac{b}{2}\rfloor } \frac{(i-1)(b-i-1)}{b-1}[B_i].$$
(ii) If $\psi_j$ denotes the cotangent class corresponding to the $j$th marked point for $j=1, \ldots, b$,
$$\sum_{j=1}^b \psi_j=\sum_{i=2}^{\lfloor \frac{b}{2}\rfloor} \frac{i(b-i)}{b-1}[B_i].$$
(iii) Let  $D=\sum_{i=2}^{\lfloor \frac{b}{2}\rfloor} c_i [B_i]$ be a divisor class with $c_i>0$ for $i=2, \ldots, \lfloor \frac{b}{2} \rfloor$. Then $D$ is big.
\end{proposition}
The third statement follows once we use that $\kappa_1$ is an ample class on $\mm_{0,b}$, thus there exists a constant $\alpha \in \mathbb Q_{>0}$ such that $D-\alpha\cdot \kappa_1$ is effective.
\begin{remark}\label{bigclasses}
A consequence of Proposition \ref{mob} is that any class on $\hh_g^k$ of the form
$$\sum_{i\geq 2}\sum_{\mu \vdash k} c_{i:\mu} [E_{i:\mu}],$$
with all coefficients $c_{i:\mu}>0$ is big.
\end{remark}

\subsection{The canonical class of $\hh_g^k$} We discuss the canonical class on the coarse moduli space $\hh_g^k$, in particular how it changes under the map $\epsilon\colon \overline{H}_g^k\rightarrow \hh_g^k$ from the stack to its coarse moduli space.

\vskip 3pt

First, in order to determine the canonical class of the Hurwitz stack one applies the Riemann-Hurwitz formula to the map $\mathfrak{b}\colon \overline{H}_g^k \rightarrow \overline{M}_{0, b}$.  Via (\ref{pb}), the ramification divisor is given by $\mbox{Ram}(\mathfrak{b})=\sum_{i, \mu\vdash k} (m(\mu)-1)[E_{i:\mu}]$, hence we obtain the following formula for the canonical class of $\overline{H}_g^k$:

\begin{equation}\label{canhur}
K_{\overline{H}_g^k}=\mathfrak{b}^*K_{\mm_{0,b}}+\mbox{Ram}(\mathfrak{b})=\sum_{i, \mu\vdash k} \left(m(\mu)\left(\frac{i(2g+2k-2-i)}{2g+2k-3}-1\right)-1\right)[E_{i:\mu}].
\end{equation}

Before our next result, we introduce some useful terminology. If $\mu$ and $\mu'$ are partitions, we write $\mu'\subseteq \mu$ when each entry of $\mu'$ appears as an entry of $\mu$ as well.

\begin{proposition}\label{cancoarse}
Assume $k\geq 3$. The canonical class of the coarse moduli space $\hh_g^k$ is given by
$$K_{\hh_g^k}=\sum_{i, \mu\vdash k} \left(m(\mu)\left(\frac{i(2g+2k-2-i)}{2g+2k-3}-1\right)-1\right)[E_{i:\mu}]-\sum_{i, \mu\vdash k} \ [E_{i:\mu}'],$$
where the second summation is taken over the boundary divisors $E_{i:\mu}'\subseteq E_{i:\mu}$ defined as the components of $E_{i:\mu}$ with a  generic point parametrizing an admissible cover whose source has an irreducible component mapping $2:1$ onto the base and a branch point at the unique node in the base.
\end{proposition}

\begin{proof}
We begin by making the following elementary observation. Suppose $u \colon Y\rightarrow \PP^1$ is a finite cover from a smooth curve $Y$ such that \emph{at most one} of its branch points is not simple. Assume $\phi\colon Y\rightarrow Y$ is an automorphism such that $u\circ \phi=u$. Then necessarily $\mbox{deg}(u)=2$ and $\phi$ is the involution of $Y$ changing the sheets of $u$. Indeed, considering the covering map $\pi\colon Y\rightarrow Y'$ where $Y':=Y/\langle \phi\rangle$, for each $y\in Y$ the ramification indices at all points in the fibre $\pi^{-1}(\pi(y))$ are the same, precisely equal to $\bigl|\mathrm{Stab}_{\langle \phi \rangle}(y)\bigr|$. Applying the Hurwitz-Zeuthen formula $u$ must have at least two branch points, one of which is necessarily simple by assumption. This implies $\mbox{deg}(u)=2$.

\vskip 3pt

Assume $t=\bigl[f\colon C=C_1\cup C_2\rightarrow R_1\cup R_2,  p_1, \ldots, p_b]$ is a general point of a component of a boundary divisor $E_{i:\mu}$ admitting a non-trivial automorphism $\phi\colon C\rightarrow C$ with $f\circ \phi=f$. Since $\mbox{deg}(f)\geq 3$, applying the previous observation, there exists one component of $R$, say $R_1$, such that $\phi_{|C_1}=\mbox{id}_{C_1}$. Furthermore, $C_2$ splits into connected components $Y_1, \ldots, Y_a$ and $Y_1', \ldots, Y_{r}'$, where $a\leq \frac{k}{2}$, such that $Y_j$ maps with degree $2$ onto $R_2$ for $j=1, \ldots, a$, whereas  each  of the components $Y_1', \ldots, Y_{r}'$ map onto $R_2$. Furthermore $\phi_{| Y_j'}=\mbox{id}_{Y_j'}$ for $j=1, \ldots, r$. Note that $\mbox{Aut}(t)=\bigl(\mathbb Z/2\mathbb Z)^{\oplus a}$. The normalization map $\overline{H}_g^k\rightarrow \overline{HM}_g^k$ has $2^{a-1}$ sheets
over the point corresponding to $t$, each of them ramified with order $2$. We denote these points by $[t, \zeta_1, \ldots, \zeta_a]\in \hh_g^k$, where $\zeta_1, \ldots, \zeta_a \in \{1,-1\}$. Using the local description provided by (\ref{localringtwist}), we conclude that the map $\epsilon \colon \overline{H}_g^k\rightarrow \hh_g^k$ is ramified with order $2$ over each such component of a divisor $E_{i:\mu}$, where $\mu\supseteq (2^a, 1^{k-2a})$.
\end{proof}

\section{Divisors on Hurwitz space via ramification and antiramification points}

We set $[n]:=\{1, \ldots, n\}$ and recall that for an integer $i\geq 0$ and a set $S\subseteq [n]$, we denote by $\delta_{i:S}$ the divisor class corresponding to the boundary divisor $\Delta_{i:S}$ on $\mm_{g,n}$, whose generic point is a transversal union of two smooth curves of genera $i$ and $g-i$ respectively, the marked points labeled by $S$ being precisely those lying on the genus $i$ component. We set as usual $\delta_{i:s}:=\sum_{|S|=s}\delta_{i:S}$.

\vskip 3pt

An important role in our work is played by (the pullbacks) of the divisor $\mathfrak{Log}$ on $\mm_{g,g}$ defined as the closure of the locus of  smooth pointed curves $[C,x_1, \ldots, x_g]\in \cM_{g,g}$ for which $h^0\bigl(C,\OO_C(x_1+\cdots+x_g)\bigr)\geq 2$. The class of this divisor was computed by Logan~\cite{Log}:
$$[\mathfrak{Log}]=-\lambda+\sum_{i=1}^g\psi_i-\sum_{i=0}^{\lfloor \frac{g}{2}\rfloor} \sum_{s=0}^g {|s-i|+1 \choose 2}   \delta_{i:s}. $$
We obtain an effective divisor $\mathfrak{Log}_{g,n}$ on $\Mbar{g}{n}$ for $n\geq g$ by averaging the pullback of $\mathfrak{Log}$ under the all choices of forgetful morphisms to $\Mbar{g}{g}$ and normalising the $\psi$-coefficient:
\begin{lemma}\label{logaverage}
The class of the effective divisor $\mathfrak{Log}_{g,n}$ in $\Mbar{g}{n}$ for $n\geq g$ is given by
\begin{equation*}
\mathfrak{Log}_{g,n}=\sum_{i=1}^n \psi_i-\frac{n}{g}\lambda-\sum_{i,s} b_{i:s}\delta_{i:s}\in CH^1(\mm_{g,n}),\ \ \mbox{ for}
\end{equation*}
\begin{equation*}
b_{i:s}=\begin{pmatrix} n-1\\ g-1\end{pmatrix}^{-1}\sum_{j=0}^s\begin{pmatrix} s\\j \end{pmatrix}\begin{pmatrix} n-s\\ g-j \end{pmatrix}\begin{pmatrix} |j-i|+1\\ 2 \end{pmatrix}
\end{equation*}
where we use the convention that $\begin{pmatrix} x\\y \end{pmatrix}=0$
for $y>x$.
\end{lemma}

\begin{proof}
This class is obtained as
$$\mathfrak{Log}_{g,n}=\begin{pmatrix} n-1\\ g-1 \end{pmatrix}^{-1}\sum_{\footnotesize{\begin{array}{cc}S\subseteq [n]\\|S|=g\end{array}}}\pi_S^*\bigl(\mathfrak{Log}\bigr),$$
where $\pi_S\colon \Mbar{g}{n}\longrightarrow \Mbar{g}{g}$
forgets all marked points outside $S\subseteq [n]$ with $|S|=g$. The claimed formula now follows by repeatedly applying the formulas \cite[Lemma 1.2]{AC2} concerning the pullbacks of the tautological generators of $CH^1(\mm_{g,n})$ under the maps $\pi_S$.
\end{proof}

\vskip 3pt

For integers $g, k\geq 2$ and $b\geq 1$, we introduce the following moduli space
$$\Mbar{g}{b+b[k-2]}:=\Mbar{g}{(k-1)b}/\mathfrak{S}_{k-2}^b,$$
where the $i$th copy of the symmetric group $\mathfrak{S}_{k-2}$ acts by permuting the marked points $x_j$ for
$j\in \mathcal{A}_i:=\bigl\{b+(i-1)(k-2)+1,\dots, b+i(k-2)\bigr\}$. Let
$$\rho\colon \Mbar{g}{b(k-1)}\longrightarrow  \Mbar{g}{b+b[k-2]}$$
be the natural projection. We refer to the marked points in $\mathcal{A}_i$ as the \emph{$i$th antiramification  points}. In the case $b=2g+2k-2$, this terminology is explained by the existence of the regular map
\begin{equation}\label{mapantiram}
\theta\colon \hh_g^k\longrightarrow \Mbar{g}{b+b[k-2]},
\end{equation}
that forgets an admissible cover $[f\colon C\rightarrow R, p_1, \ldots, p_b]$ and recalls the source curve, the (unique!) ramification point $x_i\in f^{-1}(b_i)$ and the set $A_i$ of $k-2$ unordered residual points in $f^{-1}(p_i)\setminus \{x_i\}$, for each $i=1, \ldots, b$.  We refer to the marked points in $A_i$ as being the $i$-th \emph{antiramification points} of $f$.

Observe that the ramification locus of $\rho$ is the divisor
$$\mathfrak{R}=\sum_{i=1}^b\sum_{\tiny{\begin{array}{lll}|S|=2\\S\subseteq \mathcal{A}_i\end{array}}}\delta_{0:S}$$
and denote the branch divisor by $\mathfrak{B}$.

\vskip 4pt

We now discuss the structure of the boundary divisors on $\mm_{g,b+b[k-2]}$. For any subset $T\subseteq [b]$ and $0\leq j_i\leq k-2$ for $i=1, \ldots, b$, we define $\delta_{i:T,[j_1,\dots,j_b]}\in CH^1(\Mbar{g}{b+b[k-2]})$ to be the class of the closure of the locus of stable curves with a separating node such that one component is of genus $i$ and contains the marked points labeled by $T$ and precisely $j_i$ of the $i$th antiramification points for $i=1, \ldots, b$. We denote
$$\widetilde{\delta}_{i:s}:=\sum_{\tiny{|T|+j_1+\cdots+j_b=s}}\delta_{i:T,[j_1,\dots,j_b]}\in CH^1\bigl(\mm_{g, b+b[k-2]}\bigr).$$

Let $\psi:=\sum_{i=1}^b \psi_i \in CH^1\bigl(\mm_{g, b+b[k-2]}\bigr)$  be the cotangent class corresponding to the ramification points.
Finally, we introduce the cotangent class of the antiramification points $$\Psi:=\sum_{i=1}^b\psi_{[i]}\in CH^1(\Mbar{g}{b+b[k-2]}),$$ where $\psi_{[i]}\in CH^1(\mm_{g,b+b[k-2]})$ is the class characterized by the fact $\rho^*\bigl(\psi_{[i]}\bigr)=\sum_{j\in \mathcal{A}_i} \psi_{x_j}$.

\vskip 4pt

The divisor that is of primary interest to us is the push-forward of $\mathfrak{Log}_{g, b(k-1)}$ under $\rho$,
which (after normalising the $\psi$ and the $\Psi$-coefficients) we denote by $\widetilde{\mathfrak{D}}$. The proof of the following fact is a simple application of the discussion above.

\begin{proposition}\label{classD}
The class of $\widetilde{\mathfrak{D}}$ in $CH^1(\Mbar{g}{b+b[k-2]})$ is given by
\begin{equation*}
\widetilde{\mathfrak{D}}=\psi+\Psi-\frac{b(k-1)}{g}\lambda-\sum_{i, s} c_{i:s}\widetilde{\delta}_{i:s}+a\mathfrak{B}
\end{equation*}
for
\begin{equation*}
c_{i:s}=\begin{pmatrix} b(k-1)-1\\ g-1\end{pmatrix}^{-1}\sum_{j=0}^s\begin{pmatrix} s\\j \end{pmatrix}\begin{pmatrix} b(k-1)-s\\ g-j \end{pmatrix}\begin{pmatrix} |j-i|+1\\ 2 \end{pmatrix}
\end{equation*}
and some $a>0$.
\end{proposition}

\begin{remark} Regarding the coefficients $c_{0:s}$ that will play an important role in our considerations, for all $s\geq 2$ the following inequality holds:
\begin{equation}\label{cos}
c_{0:s}=s+{s\choose 2}\frac{g-1}{b(k-1)-1}>s.
\end{equation}
For $i>0$, we will make use of the following estimate of the coefficient $c_{i:s}$, which is obtained by ignoring the absolute values in the summand in its definition in Proposition \ref{classD}:
\begin{equation}\label{cis}
c_{i:s}\geq \frac{i(i-1)b^2(k-1)^2-(i-1)(2gs+i)b(k-1)+gs(gs-g+2i-s-1)}{2g(b(k-1)-1)}.
\end{equation}
\end{remark}

In what follows we consider the pull-back of $\widetilde{\mathfrak{D}}$ to the Hurwitz space under the map $\theta$ considered in (\ref{mapantiram}).

\begin{proposition}\label{prop:LoganEff}
The pullback  $\widetilde{\mathfrak{L}}:=\theta^*\widetilde{\mathfrak{D}}$ is an effective divisor on $\hh_g^k$ for all $g\geq 2$ and $k\geq 3$.
\end{proposition}

\begin{proof}
We identify an admissible cover for each irreducible component of $\widetilde{\mathfrak{D}}$ that lies outside of the pullback.

Consider the admissible cover constructed by gluing a genus $g$ hyperelliptic double cover $h\colon C\to R_1\cong\PP^1$ at an unramified point $p\in C$ to a simply branched degree $k-1$ rational cover $u \colon C_2\to R_2 \cong\PP^1$ at an unramified point which we also denote by $p\in C_2$, and further attaching the required rational tails mapping isomorphically to $R_1$ at the $k-2$ points $u^{-1}(u(q))\setminus \{q\}$ and a rational tail mapping isomorphically to $R_2$ at the point conjugate to $p$ under the hyperelliptic involution. The ordering of the branch points will be specified below.

\vskip 3pt

Recall the class of $\widetilde{\mathfrak{D}}$ is the pushforward of
$$\mathfrak{Log}_{g,n}=\begin{pmatrix} n-1\\ g-1 \end{pmatrix}^{-1}\sum_{\footnotesize{\begin{array}{cc}S\subseteq [n]\\|S|=g\end{array}}}\pi_S^*\bigl(\mathfrak{Log}\bigr),$$
where $\pi_S\colon \Mbar{g}{n}\longrightarrow \Mbar{g}{g}$
forgets all marked points outside $S\subseteq [n]$ with $|S|=g$.

Hence the irreducible components of the divisor $\widetilde{\mathfrak{D}}$ in $\mm_{g,b+b[k-2]}$ are indexed by partitions $[T,j_1,\ldots, j_b]$ for $T\subseteq [b]$ and $0\leq j_i\leq k-2$ for $i=1, \ldots, b$ and
$$|T|+j_1+\cdots+j_b=g.$$ For such a partition, if $A_{j_i}\subseteq A_i$ is a subset of antiramification points with $|A_{j_i}|=j_i$ and $S=T\cup A_{j_1}\cup \ldots \cup A_{j_b}\subseteq [b(k-1)]$, then the general point of the corresponding component of $\widetilde{\mathfrak{D}}$ corresponds to a pointed curve satisfying $h^0\bigl(C, \OO_C(\sum_{i\in S} x_i)\bigr)=2$.
For each such partition we specify a labelling of the above constructed admissible cover such that the admissible cover lies outside of the pullback of the specified irreducible component of $\widetilde{\mathfrak{D}}$. Let
$$Z:=\bigl\{i\in [b]: j_i>0 \mbox{ and } i\notin T\bigr\}$$
and let $r:=|Z|$ and $a:=|T|$. Label $a+r\leq g$ of the $2g+2$ branch points of points of $h\colon C\to\PP^1$ as the points $T\cup Z$ and choose a fixed labelling of the remaining branch points of the admissible cover. Observe that as $C$ is hyperelliptic, we have
$$h^0\bigl(C,\OO_C(w_1+\cdots+w_{a}+(g-a)p)\bigr)=1$$
for any choice of $a$ distinct Weierstrass points $w_i$ of $C$. Hence this admissible cover is not contained in the pullback of the irreducible component of $\widetilde{\mathfrak{D}}$ specified by the partition $[T,j_1,\dots,j_b]$.
\end{proof}

Before stating our next result, we recall the partition $[b(k-1)]=[b]\cup \mathcal{A}_1\cup \ldots \cup \mathcal{A}_b$ of the set of labels for the ramification and antiramification points respectively.

\begin{proposition}\label{psiclasses}
At the level of divisors, the map $\theta\colon \hh_g^k\rightarrow \mm_{g,b+b[k-2]}$ behaves as follows:
$$(i) \ \ \ \theta^*(\psi)=\sum_{i=2}^{g+k-1}\sum_{\mu\vdash k} m(\mu)\frac{i(b-i)}{2(b-1)}[E_{i:\mu}],$$ \
$$ (ii) \ \ \ \theta^*(\Psi)= (k-2)\sum_{i=2}^{g+k-1}\sum_{\mu \vdash k} m(\mu)\frac{i(b-i)}{b-1}[E_{i:\mu}]+\sum_{s=2}^{b-2}\sum_{i=1}^b
\sum_{\tiny{\begin{array}{lll}S\subseteq [b(k-1)]\\ \ \ |S\cap \mathcal{A}_i|=s\end{array}}}  s\ \theta^*\bigl(\delta_{0: S}\bigr).$$
\end{proposition}

\begin{proof}
We recall that we have introduced in (\ref{eq:univcurve}) the universal degree $k$ admissible cover
$$f\colon \cC\rightarrow \P$$
and we denote by $\varphi\colon \P\rightarrow \hh_g^k$ and $v:=\varphi\circ f\colon \cC\rightarrow \hh_g^k$ the two universal curves over the Hurwitz space. We consider the ramification divisors $R_1, \ldots, R_b\subseteq \cC$, as well as the antiramification divisors
$A_1, \ldots, A_b\subseteq \cC$. If $\mathfrak{B}_1, \ldots, \mathfrak{B}_b\subseteq \P$ denote the corresponding branch divisors of $f$, then clearly $f_*([R_i])=[\mathfrak{B}_i]$ and $f_*([A_i])=(k-2)[\mathfrak{B}_i]$. It is important to observe that $R_i\cdot A_i=0$ for $i=1, \ldots, b$.

\vskip 3pt

In order to estimate the class $\theta^*(\psi_i)$, we multiply the relation
\begin{equation}\label{eq:ram3}
f^*(\mathfrak{B}_i)=2R_i+A_i
\end{equation}
with the class of $R_i$, and using that $R_i$ and $A_i$ are disjoint we write as follows:
$$\theta^*(\psi_i)=-v_*\bigl([R_i]^2\bigr)=-\frac{1}{2} v_*\bigl(f^*(\mathfrak{B}_i)\cdot R_i\bigr)=-\frac{1}{2}\varphi_*\bigl([\mathfrak{B}_i]^2\bigr)=\frac{1}{2} \mathfrak{b}^*\bigl(\psi_i^{{\small{\mathrm{b}}}}\bigr),$$
where in the interest of clarity we denote by $\psi_i^{\small{\mathrm{b}}}\in CH^1(\mm_{0,b})$ the cotangent class corresponding to the $i$th branch point. Now (i) follows by applying part (ii) of Proposition \ref{mob}. To estimate $\theta^*(\psi_{[i]})$, we first introduce the class
$\widetilde{\psi}_{[i]}$ on $\mm_{g,b+b[k-2]}$ characterized by
$$\Bigl(\widetilde{\psi}_{[i]}\Bigr)_{[C, x_1, \ldots, x_b, A_1, \ldots, A_b]}=\bigotimes_{x\in A_i} T_{x}^{\vee}(C),$$
for each  $[C, x_1, \ldots, x_b, A_1, \ldots, A_b]\in \mm_{g,b+b[k-2]}$. Then using \cite[(5)]{FV}, we observe that

\begin{equation}\label{eq:antii}
\rho^*\bigl(\widetilde{\psi}_{[i]}\bigr)=\sum_{x\in A_i}\psi_x-\sum_{s=2}^{b-2}\sum_{|S\cap A_i|=s} s \delta_{0:S}.
\end{equation}
Next, we multiply (\ref{eq:ram3}) with the class of the antiramification divisor $A_i$ and write:
$$\theta^*\bigl(\widetilde{\psi}_{[i]}\bigr)=-v_*\bigl([A_i]^2\bigr)=-v_*\bigl(f^*([\mathfrak{B}_i])\cdot A_i\bigr)=-(k-2)\varphi_*\bigl([\mathfrak{B}_i]^2\bigr)=(k-2)\mathfrak{b}^*\bigl(\psi_i^{\small{\mathrm{b}}}\bigr).$$
The rest follows again via Proposition \ref{mob} (ii) coupled with formula (\ref{pb}).
\end{proof}

\section{The positivity of the canonical class of $\hh_g^k$}

We are now in a position to complete the proof of both Theorems \ref{kodhurw} and \ref{kodhurwcoarse3}. Recall that we have introduced in (\ref{mapantiram}) the map $\theta \colon \hh_g^k\rightarrow \mm_{g,b+b[k-2]}$ retaining from a cover its source together with ramification and antiramification points.

\begin{proposition}\label{pullbackantir}
The following divisor classes are effective on $\hh_g^k$:

(i) $\theta^*(\widetilde{\delta}_{0:2})-(k-2)[E_{2:(1^k)}]-(k-4)[E_{2:(2^2, 1^{k-4})}]-(k-3)[E_{2:(3,1^{k-3})}]\geq 0$.

(ii) $\theta^*(\widetilde{\delta}_{0:3})-4[E_{2:(2^2,1^{k-4})}]\geq 0$.

(iii) $\theta^*(\widetilde{\delta}_{0:4})-[E_{2:(3,1^{k-3})}]\geq 0$.

\end{proposition}

\begin{proof}
We analyse the image under $\theta$ of a general point $t$ belonging to various boundary components of $\hh_g^k$. Suppose first that
$t=[f\colon C\rightarrow R, p_1, \ldots, p_b]$ corresponds to the general point of a component of $E_{2:(1^k)}$. The base $R$ is the transverse union of two rational components $R_1$ and $R_2$ meeting at a point $p$ and assume $p_1, \ldots, p_{b-2}\in R_1\setminus \{p\}$ and $p_{b-1}, p_b\in R_2\setminus \{p\}$. Denoting by $C_i:=f^{-1}(R_i)$ and by $x_{b-1}, x_{b}\in C_2$ the ramification points over $p_{b-1}$ and $p_{b}$ respectively, we observe that $C_2$ contains $k-2$ smooth rational components each intersecting $C_1$ at one point and mapping isomorphically onto $R_2$. Each of these components contains two antiramification points of $f$ lying over $p_{b-1}$ and $p_{b-2}$ respectively. This implies that the image under $\theta$ of the boundary component of $\hh_g^k$ containing the point $t$ lies in the $(k-2)$nd self intersection of the boundary divisor $\Delta_{0:\emptyset,[0,\ldots, 0, 1,1]}$. This in turn yields that the pullback $\theta^*(\widetilde{\delta}_{0:2})$ contains $E_{2:(1^k)}$ with multiplicity at least $k-2$.

\vskip 3pt

Similarly, if $t=[f\colon C\rightarrow R, p_1, \ldots, p_b]$ is a general point of a component of $E_{2:(3,1^{k-3})}$, with $R=R_1\cup R_2$ as above and $p_1, \ldots, p_{b-2}\in R_1\setminus \{p\}$ and $p_{b-1}, p_b\in R_2\setminus \{p\}$, let $C_2$ denote the component of $C$ mapping with degree $3$ onto $R_2$. From the Hurwitz-Zeuthen formula, $C_2$ is necessarily of genus zero. The curve $C_2$ will contain two ramification points $x_{b-1}$ and $x_{b}$, as well as two antiramification points lying in the fibres over $p_{b-1}$ and $p_b$ respectively. Therefore the component of the boundary divisor of $\hh_g^k$ containing $t$ is mapped to the divisor $\Delta_{0:\{b-1,b\}, [0, \ldots, 0, 1,1]}$, that is, $\theta^*(\widetilde{\delta}_{0:3})$ contains $E_{2:(3,1^{k-3})}$. Arguing as above, $\theta(t)$ lies in the $(k-3)$rd self intersection of $\Delta_{0:\emptyset,[0,\ldots, 0, 1,1]}$.

\vskip 3pt

Finally, let $t=[f\colon C\rightarrow R, p_1, \ldots, p_b]$ be a general point of a component of $E_{2:(2^2,1^{k-4})}$, with $R=R_1\cup R_2$ and the distribution of the ramification points as above.  Then $f^{-1}(R_2)$ contains two smooth rational curves $C_2$ and $C_2'$ mapping with degree $2$ onto $R_2$ and meeting $C_1:=f^{-1}(R_1)$ at a ramification point. Assume $x_{b-1}\in C_2$ and $x_b\in C_2'$ are the two ramification points over $p_{b-1}$ and $p_b$. Then $C_2$ (respectively $C_2'$) contains two further antiramification points lying in $f^{-1}(p_b)$ (respectively $f^{-1}(p_{b-1}))$. Note furthermore that both $C_2$ and $C_2'$ admit an automorphism of order $2$ fixing the two ramification points and permuting the two respective  antiramification points. It follows that both $\theta^*\bigl(\Delta_{0:\{b\},[0, \ldots, 0, 2,0]}\bigr)$ and $\theta^*\bigl(\Delta_{0:\{b-1\}, [0, \ldots, 0, 2]}\bigr)$ contain the boundary divisor of $\hh_g^k$ containing the point $t$ with multiplicity at least $2$, hence
$$\theta^*(\widetilde{\delta}_{0:3})\geq \theta^*\bigl(\Delta_{0:\{b\},[0, \ldots, 0, 2,0]}\bigr)+ \theta^*\bigl(\Delta_{0:\{b-1\}, [0, \ldots, 0, 2]}\bigr)\geq 4[E_{2:(2^2,1^{k-4})}].$$ Finally, the point $\theta(t)$ lies in the $(k-4)$th self intersection of $\Delta_{0:\emptyset,[0,\ldots, 0, 1,1]}$.
\end{proof}

\vskip 3pt

We are now in a position to complete the proof of Theorem \ref{kodhurw} and show that for all $g\geq 2,$ and $k\geq 3$ the Kodaira-Iitaka dimension of the stack $\overline{H}^k_g$ is maximal.

\vskip 5pt

\noindent \emph{Proof of Theorem \ref{kodhurw}.} We consider once more the map $\theta\colon \hh_g^k\rightarrow \mm_{g,b+b[k-2]}$ and the divisor $\widetilde{\mathfrak{L}}$ introduced in Proposition \ref{prop:LoganEff}. We show that there exists a constant $B>0$ such that $K_{\overline{H}_g^k}-B\cdot \widetilde{\mathfrak{L}}$ is big, which implies that $K_{\overline{H}_g^k}$ itself is big. To that end, via Proposition \ref{mob} (iii), it  suffices to show that $K_{\overline{H}_g^k}-B\cdot \widetilde{\mathfrak{L}}$ has a representative in terms of boundary classes in which the coefficient of the class of each irreducible component of $E_{i:\mu}$ is \emph{positive}.

\vskip 3pt

Observe first that the image of $\theta$ is disjoint from the branch locus $\mathfrak{B}$ of $$\rho\colon \mm_{g,b(k-1)}\rightarrow \mm_{g,b+b[k-2]}.$$

Indeed, the source of an admissible cover $[f\colon C\rightarrow R, p_1, \ldots, p_b]\in \hh_g^k$ cannot contain a smooth rational component $C'$ containing precisely two antiramification points lying over the same branch point, for then $\mbox{deg}(f_{|C'})\geq 2$, which implies that $f_{|C'}$ admits further ramification at points lying in $C'\setminus C$.

\vskip 3pt

We shall find a constant $B>0$ such that for  all $2\leq i\leq \frac{b}{2}$ and all partitions $\mu\vdash k$ the following quantity, equaling the coefficient of
$[E_{i:\mu}]$ in $K_{\overline{H}_g^k}-B\cdot \widetilde{\mathfrak{L}}$, is \emph{positive}:

\begin{equation}\label{ineqstack}
m(\mu)\Bigl(\frac{i(b-i)}{b-1}-1\Bigr)-1+B m(\mu) \frac{b(k-1)}{g} \Bigl(\frac{i(b-i)}{8(b-1)}-\frac{1}{12}\Bigl(k-\frac{1}{\mu}\Bigr)\Bigr)-
\end{equation}
$$-B m(\mu) \frac{(2k-3)i(b-i)}{2(b-1)} + B \theta^*\Bigl(\sum_{j,s} c_{j:s}\widetilde{\delta}_{j:s}-\sum_{s=2}^{b-2}\sum_{j=1}^b \sum_{|S\cap \mathcal{A}_j|=s} s \delta_{0:S}\Bigr)_{i:\mu}>0,$$
where for a boundary divisor $\alpha$ on $\mm_{g,b+b[k-2]}$, we denote by $\theta^*(\alpha)_{i:\mu}$ the coefficient of $[E_{i:\mu}]$ in $\theta^*(\alpha)$ viewed as a boundary divisor. Observe that the contribution of $\widetilde{\mathfrak{L}}$ follows from Proposition~\ref{classD} and Proposition~\ref{psiclasses}. Set

\begin{equation}\label{Bvalue}
B:={\left(\frac{2(b-2)}{b-1}-2\right)\left(\frac{b(b-2)(k-1)}{4g(b-1)}-\frac{(2k-3)(b-2)}{b-1}+2(k-2)\right)^{-1}}\\
\end{equation}
$$=\frac{16g}{b^3-2b^2(g+1)+4bg-16g^2}>0.\hspace{5cm}$$

 \vskip 5pt


We check (\ref{ineqstack}) case by case, starting with the most challenging case $i=2$, since this is when the coefficient of $[E_{i:\mu}]$ in
$K_{\hh^k_g}$ may be negative.

\noindent {\bf (i)} First assume that $\mu=(1^k)$, thus $\frac{1}{\mu}=k$. Using (\ref{pullbackantir}) we have $\theta^*(\widetilde{\delta}_{0:2})\geq (k-2)[E_{2:(1^k)}]$. Furthermore, for any component $Z$ of $E_{2:(1^k)}$ we have
$$\theta(Z)\nsubseteq \sum_{s=2}^{b-2}\sum_{j=1}^b \sum_{|S\cap \mathcal{A}_j|=s} \Delta_{0:S},$$
therefore inequality (\ref{ineqstack}) becomes in this case
$$\frac{2(b-2)}{b-1}-2+B\left(\frac{b(b-2)(k-1)}{4g(b-1)}-\frac{(2k-3)(b-2)}{b-1}+c_{0:2}(k-2)\right)>0,$$
which is clear by our choice of $B$ and the observation that $c_{0:2}>2$ by (\ref{cos}).

\vskip 5pt

\noindent {\bf (ii)} Assume now  $\mu=(3,1^{k-3})$, thus $\frac{1}{\mu}=k-\frac{8}{3}$ and $m(\mu)=3$. Using that for any component $Z$ of $E_{2:(3,1^{k-3})}$ one has
$$\theta(Z)\nsubseteq \sum_{s=2}^{b-2}\sum_{j=1}^b \sum_{|S\cap \mathcal{A}_j|=s} \Delta_{0:S},$$
and $\theta^*(\widetilde{\delta}_{0:4})\geq [E_{2:(3,1^{k-3})}]$ and $\theta^*(\widetilde{\delta}_{0:2})\geq (k-3)[E_{2: (3,1^{k-3})}]$ by (\ref{pullbackantir}), inequality (\ref{ineqstack}) is implied by the following inequality

$$3\Bigl(\frac{2(b-2)}{b-1}-1\Bigr)-1+B\Bigl(\frac{3b(b-2)(k-1)}{4g(b-1)}-\frac{2}{3}-\frac{3(b-2)(2k-3)}{b-1}+c_{0:2}(k-3)+c_{0:4}\Bigr)>0.$$
Let $a:=k-3\geq 0$, thus $b=2g+2a+4$. After substituting $B$ by the value provided by (\ref{Bvalue}) and observing  $c_{0:s}>s$ by (\ref{cos}) we obtain that the above value is greater than
$$\frac{2(3ag^2+6a^2g+9ag+10g+3a^3+15a^3+24a+12)}{3(ag^2+2a^2g+7ag+6g+a^3+5a^2+8a+4)} >0.$$

\vskip 5pt

\noindent {\bf (iii)} Suppose $\mu=(2^2,1^{k-4})$, thus $\frac{1}{\mu}=k-3$ and $m(\mu)=2$. In particular, $a\geq 1$. In this case (\ref{pullbackantir}) provides
that $\theta^*(\widetilde{\delta}_{0:3})\geq 4[E_{2:(2^2,1^{k-4})}]$ and $\theta^*(\widetilde{\delta}_{0:2})\geq (k-4)[E_{2:(2^2,1^{k-4})}]$. On the other hand, the discussion of the final part of the Proof of Proposition \ref{pullbackantir} shows that
$$\mbox{ord}_{E_{2:(2^2, 1^{k-4})}}\theta^*\left(\sum_{s=2}^{b-2}\sum_{j=1}^b \sum_{|S\cap \mathcal{A}_j|=s} s \Delta_{0:S}\right)=8$$
the statement holds along each irreducible component of $E_{2:(2^2,1^{k-4})}$. It follows that (\ref{ineqstack}) would be implied by the following inequality:

$$\frac{4(b-2)}{b-1}-3+B\left(\frac{b(b-2)(k-1)}{2(b-1)g}-\frac{1}{2}- \frac{2(2k-3)(b-2)}{b-1}+c_{0:2}(k-4)+4c_{0:3}-8\right)>0.$$

By substituting the value for $B$ and observing that $c_{0:s}>s$ the above value is greater than
$$\frac{ag^2+2a^2g+3ag+g+a^3+5a^2+8a+4}{ag^2+2a^2g+7ag+6g+a^3+5a^2+8a+4}>0 $$

\vskip 4pt

\noindent {{\bf{(iv)}} Assume now that $i\geq 3$ and assume first $\mu\neq (1^k)$, that is, $m(\mu)\geq 2$. We shall show that the following inequality, which is stronger than (\ref{ineqstack}), holds:

\begin{equation}\label{m2}
m(\mu)\left( \Bigl(\frac{i(b-i)}{b-1}-1\Bigr)+B \frac{b(k-1)}{g} \Bigl(\frac{i(b-i)}{8(b-1)}-\frac{1}{12}\Bigl(k-\frac{1}{\mu}\Bigr)\Bigr)-
B \frac{(2k-3)i(b-i)}{2(b-1)} \right)-1>0.
\end{equation}

The coefficient of $i(b-i)$ in this expression being positive, its smallest value is attained when $i=3$. Furthermore, using the means inequality we find that
$\frac{1}{\mu}\geq \frac{1}{k}$, whereas $m(\mu)\geq 2$ for any partition $\mu$ of $k$.  Substituting $i=3$, $m(\mu)=2$ and  $\frac{1}{\mu}=\frac{1}{k}$ in (\ref{m2}), we obtain the inequality which implies (\ref{m2}), which in turn implies inequality (\ref{ineqstack}):

$$\frac{3(ag^2+2a^2g+2ag+2g+a^3+4a^2+4a)}{ag^2+2a^2g+7ag+6g+a^3+5a^2+8a+4} >0.$$

\vskip 4pt

{{\bf (v)}} Assume now that $i\geq 3$ and $\mu=(1^k)$, in which case as $\mu$ is the cycle class of an even permutation, $i$ is even and hence $i\geq 4$. In this case (\ref{m2}) can be reduced to the following, obtained by substituting $i=4$:

$$\frac{2(ag^2+2a^2g +ag+2g+a^3+3a^2-4)}{ag^2+2a^2g+7ag +6g+a^3+5a^2+8a+4}>0,$$
\vskip 4pt
\noindent which holds in all cases outside of $g=2$ and $k=a+3=3$ where the left hand side is equal to zero. However, in this case as all other inequalities hold, the choice of
$$B=\frac{16g}{b^3-2b^2(g+1)+4bg-16g^2}-\varepsilon=\frac{1}{4}-\varepsilon$$
for $\varepsilon>0$ small enough completes the proof.
\hfill $\Box$

\vskip 4pt

We are now in a position to prove Theorem \ref{kodhurwcoarse3} and show that all coarse moduli spaces $\hh_g^3$ of trigonal curves of genus $g\geq 2$ have maximal Kodaira-Iitaka dimension.

\vskip 5pt

\noindent \emph{Proof of Theorem \ref{kodhurwcoarse3}.} We use the constant $B$ introduced in (\ref{Bvalue}), which in the case $k=3$ takes the form $$B=\frac{g}{3g+2}$$ and we show that $K_{\hh_g^3}-B\cdot \widetilde{\mathfrak{L}}$  admits a boundary representative in which all boundary components of $\hh_g^3$ appear with positive coefficients. Using Proposition \ref{cancoarse}, we have
$$K_{\hh_g^3}=K_{\overline{H}_g^3}-\sum_{h=1}^g[E_{2h+1:(2,1)}^\text{hyp}],$$
where the general point of each component of $E_{2h+1:(2,1)}^\text{hyp}$ parametrize an admissible cover $t=[f\colon C\rightarrow R=R_1\cup_p R_2, p_1, \ldots, p_{2g+4}]$, where $C_1:=f^{-1}(R_1)$ is a smooth curve of genus $g-h$ mapping with degree $3$ over $R_1$ and $f^{-1}(R_2)=C_2\cup C_2'$, where $C_2$ is a smooth hyperelliptic curve of genus $h$ mapping with degree $2$ onto $R_2$ and meeting $C_1$ and precisely one point $q\in f^{-1}(p)$, which is a ramification point for both $C_1$ and $C_2$. The component $C_2'$ is a smooth rational curve mapping isomorphically onto $R_2$.

\vskip 3pt

Observe that $\theta^*\bigl(\widetilde{\delta}_{0:2h+1}\bigr)\geq [E_{2h+1:(2,1)}^\text{hyp}]$, as well as $\theta^*\bigl(\widetilde{\delta}_{h:2h+1}\bigr)\geq [E_{2h+1:(2,1)}^\text{hyp}]$. Using the estimate (\ref{cis}), the coefficient of $[E_{2h+1:(2,1)}^\text{hyp}]$ in the expression of
$K_{\hh_g^3}-B\cdot \widetilde{\mathfrak{L}}$ is at least equal to
$$\frac{4g^2h^2+38g^2h-20gh^2-28g^2+77gh-28h^2-65g+28h-28}{(4g+7)(3g+2)}>0,$$
which can be checked in a straightforward manner. Since the other boundary coefficients of $K_{\hh_g^3}$ and $K_{\overline{H}_g^3}$ coincide, we can invoke the proof of Theorem \ref{kodhurw} to conclude.

\hfill $\Box$

\vskip 5pt

As described in the introduction, when $k\geq \frac{g+2}{2}$, one has a natural map
$$\chi\colon \hh_g^k\rightarrow \mm_{g,2k-g-2},$$
which assigns to an admissible cover $\bigl[f\colon C\rightarrow R, \ p_1, \ldots, p_{b}\bigr]$ the stabilization of the pointed curve
$\bigl[C, q_1, \ldots, q_{2k-g-2}]$, where $q_i\in C$ is the unique ramification point of $f$ lying in $f^{-1}(p_i)$, for $i=1, \ldots, 2k-g-2$.

\begin{proposition}\label{kodram}
The map $\chi\colon \hh_g^k\rightarrow \mm_{g,2k-g-2}$ is generically finite. It follows that $\hh_g^k$ is a variety of general type when $k\geq g+1$ and $g\geq 12$. Similarly, $\hh_g^k$ is of general type when $k\geq \frac{g+2}{2}$ and $g\geq 22$.
\end{proposition}
\begin{proof}
The generic finiteness of the map $\chi$ follows essentially from results in \cite{EH}. We set $n:=2k-g-2\geq 0$ and consider the stable curve $[X, q_1, \ldots, q_n]\in \mm_{g,n}$, where $X$ consists of a smooth rational component $R$ and $g$ elliptic tails $E_1, \ldots, E_g$, each $E_i$ meeting $R$ at a single point $x_i$. The marked points $q_1, \ldots, q_n$ lie on $R\setminus\{x_1, \ldots, x_g\}$. Then the fibre $\chi^{-1}\bigl([X, q_1, \ldots,q_n]\bigr)$ is isomorphic to the variety of limit linear series of type $\mathfrak{g}^1_k$ on $X$ having simple ramification at each of the points $q_i$. Applying \cite[Theorem 1.1]{EH}, we obtain that this variety is pure of dimension
$$\rho(g,1,k)-n=g-2(g-k+1)-(2k-g-2)=0.$$
Therefore $\chi$ is generically finite, in particular $\kappa(\hh_g^k)\geq \kappa(\mm_{g,n})$. When $g\geq 12$ and $k\geq g+1$, then $n\geq g+1$ and it follows from \cite{Log} that $\mm_{g,n}$ is of general type in this range, which finishes the proof.
\end{proof}

\vskip 4pt

\begin{remark} Note that in the range $k\geq g+1$, one has an $\mathfrak{S}_b$-cover $\hh_g^k\rightarrow \hh_{g,k}$ whose source variety $\hh_g^k$ is of general type, whereas its base $\hh_{g,k}$ is uniruled. Observe also that the degree of $\chi\colon \hh_g^k\rightarrow \mm_{g,2k-g-b}$ is the Catalan number $\frac{(2k-2)!}{k!\cdot (k-1)!}$, therefore independent of $g$!
\end{remark}

\vskip 8pt

\small{\noindent {{\bf{Acknowledgments:}}} The first author has been supported by the DFG Grant \emph{Syzygien und Moduli} and by the ERC Advanced Grant SYZYGY. The second author was supported by the Alexander von Humboldt Foundation during the preparation of this article. This project has received funding from the European Research Council (ERC) under the European Union Horizon 2020 research and innovation program (grant agreement No. 834172).}

\end{document}